\documentclass[a4paper]{amsart}
\usepackage{amssymb}
\usepackage{xypic}
\usepackage{graphicx}
\usepackage{epstopdf}

\newtheorem{theorem}{Theorem}[section]
\newtheorem{lemma}[theorem]{Lemma}
\newtheorem{sublemma}[theorem]{Sublemma}
\newtheorem{cor}[theorem]{Corollary}

\theoremstyle{definition}
\newtheorem{definition}[theorem]{Definition}

\theoremstyle{remark}
\newtheorem{remark}[theorem]{Remark}

\numberwithin{equation}{section}

\newcommand{\C}{\mathbb{C}}
\newcommand{\N}{\mathbb{N}}
\newcommand{\Z}{\mathbb{Z}}
\newcommand{\R}{\mathbb{R}}
\newcommand{\HH}{\mathbb{H}}
\newcommand{\OO}{\mathbb{O}}

\newcommand{\mft}{\mathfrak{t}}
\newcommand{\mfg}{\mathfrak{g}}
\newcommand{\mfk}{\mathfrak{k}}

\DeclareMathOperator{\diag}{diag}

\DeclareMathOperator{\rank}{rank}
\DeclareMathOperator{\id}{Id}
\DeclareMathOperator{\odd}{odd}
\DeclareMathOperator{\Spin}{Spin}

\DeclareMathOperator{\SU}{SU}
\DeclareMathOperator{\Sp}{Sp}
\DeclareMathOperator{\U}{U}

\title{Positively curved GKM-manifolds}

\author{Oliver Goertsches}
\address{Mathematisches Institut der Universit\"at M\"unchen \\ Theresienstrasse 39 \\ D-80333 M\"unchen\\ Germany}
\email{goertsches@math.lmu.de}
\thanks{}

\author{Michael Wiemeler}
\address{Institut f\"ur Mathematik\\ Universit\"at Augsburg\\ D-86135 Augsburg\\ Germany}
\email{michael.wiemeler@math.uni-augsburg.de}
\thanks{}



\begin{document}
\begin{abstract}
 Let \(T\) be a torus of dimension \(\geq k\) and \(M\) a \(T\)-manifold.
 \(M\) is a GKM$_k$-manifold if the action is equivariantly formal, has only isolated fixed points, and any $k$ weights of the isotropy representation in the fixed points are linearly independent. 

 In this paper we compute the cohomology rings with real and integer coefficients of GKM$_3$- and GKM$_4$-manifolds which admit invariant metrics of positive sectional curvature.
\end{abstract}

\maketitle


\section{Introduction}
\label{sec:introduction}

The classification of positively curved manifolds is a longstanding problem in Riemannian geometry.
So far only few examples of such manifolds are known.
In dimension greater than 24 all such examples are diffeomorphic to compact rank one symmetric spaces (CROSSs).
All examples admit non-trivial \(S^1\)-actions
and it is conjectured that all positively curved manifolds admit a non-trivial \(S^1\)-action.
Based on this conjecture several authors (see for example \cite{MR1255926}, \cite{MR2051400}, \cite{MR2139252}, \cite{amann13:_topol} and others) have considered the above classification problem under the extra assumption that there is a isometric circle or torus action.
Usually a lower bound on the  dimension of the acting torus is assumed which grows with the dimension of the manifold.

In this paper we study positively curved manifolds $M$ which admit an isometric torus action of GKM-type \cite{MR1489894}.
These actions have only finitely many fixed points (hence, $M$ is necessarily even-dimensional) and the union of all one-dimensional orbits is two-dimensional.
These assumptions on the strata of the action are stronger than what is usually assumed, but our methods work for three- and four-dimensional tori, independent of the dimension of $M$. In the appendix we observe that all known examples of positively curved manifolds in even dimensions admit isometric GKM actions. 

Our main result states that under a slightly more restrictive assumption on the strata we can determine the real cohomology ring of $M$:

\begin{theorem}\label{sec:main-results-8}
 Let $M$ be a compact connected positively curved orientable Riemannian manifold satisfying $H^{\odd}(M;\R)=0$. 
 \begin{enumerate}
 \item Assume that $M$ admits an isometric torus action of type GKM$_4$, i.e., an action with finitely many fixed points such that at each fixed point, any four weights of the isotropy representation are linearly independent. Then $M$ has the real cohomology ring of $S^{2n}$ or $\C P^n$.
 \item  Assume that $M$ admits an isometric torus action of type GKM$_3$, i.e., an action with finitely many fixed points such that at each fixed point, any three weights of the isotropy representation are linearly independent. 
 Then \(M\) has the real cohomology ring of a compact rank one symmetric space.
\end{enumerate}  
  In both cases, the real Pontryagin classes of \(M\) are standard,
  i.e. there exists an isomorphism of rings
  \(f:H^*(M;\mathbb{R})\rightarrow H^*(K;\mathbb{R})\) which preserves
  Pontryagin classes, where \(K\) is
  a compact rank one symmetric space.
\end{theorem}

As a corollary we get the following:

\begin{cor} 
\label{sec:main-results-9} 
    Let \(M\) be a compact connected positively curved orientable Riemannian manifold with $H^{\odd}(M;\R)=0$ which admits an isometric torus action of type GKM$_3$ and an invariant almost complex structure. 
 Then \(M\) has the real cohomology ring of a complex projective space.
\end{cor}

Under even stronger conditions on the weights of the isotropy representations (see Section~\ref{sec:integer-coefficients}) we can also prove versions of the above theorem for cohomology with integer coefficients.
By combining these versions with results from rational homotopy theory we can conclude that up to diffeomorphism there are only finitely many GKM manifolds which admit invariant metrics of positive sectional curvature and satisfy these stronger conditions, see Remark \ref{rem:rathomtop}.

We also have a version of Theorem~\ref{sec:main-results-8} for non-orientable manifolds (see Corollary~\ref{sec:non-orientable-gkm}).

This paper is organized as follows.
In Section~\ref{sec:gkm-theory} we review the basics of GKM theory and give the precise GKM condition.
In Section~\ref{sec:main-results-7} we give an outline of the proofs of the main results.
In Section \ref{sec:rankoneactions} we describe the GKM graph of natural torus actions on CROSSs.
In Section \ref{sec:main} we prove our main results.
In Sections \ref{sec:integer-coefficients} and~\ref{sec:non-oriented} we discuss extensions of our main results to cohomology with integer coefficients and to non-orientable manifolds.
In the appendix we describe the GKM graphs of natural torus actions on the known nonsymmetric examples of positively curved manifolds.\\

\noindent {\it Acknowledgements.} We are grateful to Augustin-Liviu Mare for some explanations on the octonionic flag manifold $F_4/\Spin(8)$, and the anonymous referees for careful reading and various improvements resulting from their comments.

\section{GKM theory}
\label{sec:gkm-theory}

Throughout this paper, we consider an action of a compact torus $T$ on a compact differentiable manifold $M$, which we will always assume to be connected. Then its equivariant cohomology is defined as 
\[
H^*_T(M;R)=H^*(M\times_T ET;R),
\]
where $R$ is the coefficient ring and $ET\to BT$ the classifying bundle of $T$. For the moment we will consider only the real numbers as coefficient ring and thus omit the coefficients from the notation, but we will also consider the case of the integers in Section \ref{sec:integer-coefficients} below. The equivariant cohomology $H^*_T(M)$ has, via the projection $M\times_T ET\to BT$, the natural structure of an $H^*(BT)=S(\mft^*)$-algebra.

We say that the $T$-action is \emph{equivariantly formal} if $H^*_T(M)$ is a free $H^*(BT)$-module. For such actions, the ordinary (de Rham) cohomology ring of $M$ can be computed from the equivariant cohomology algebra because of the following well-known statement \cite[Theorem 3.10.4  and Corollary 4.2.3]{MR1236839}: If the $T$-action on $M$ is equivariantly formal, then the natural map $H^*_T(M)\to H^*(M)$ is surjective and induces a ring isomorphism
\[
H^*_T(M)/(H^{>0}(BT))\cong H^*(M).
\]
For torus actions with only finitely many fixed points this condition is also equivalent to the fact that \(H^{\text{odd}}(M)=0\). In this case it follows that the number of fixed points of the torus action, which is equal to the Euler characteristic of $M$, is given by the total dimension of $H^*(M)$.

The Borel localization theorem \cite[Corollary 3.1.8]{MR1236839} implies that the canonical restriction map 
\[
H^*_T(M)\longrightarrow H^*_T(M^T)
\]
has as kernel the $H^*(BT)$-torsion submodule of $H^*_T(M)$; hence, for equivariantly formal actions, this map is injective, and one can try to compute $H^*_T(M)$ by understanding its image in $H^*_T(M^T)$. 

The Chang-Skjelbred Lemma \cite[Lemma 2.3]{MR0375357} describes this image in terms of the \emph{one-skeleton} 
\[
M_1=\{p\in M\mid \dim T\cdot p \leq 1\}.
\] of the action: if the $T$-action is equivariantly formal, then the sequence
\[
0\longrightarrow H^*_T(M) \longrightarrow H^*_T(M^T) \longrightarrow H^*_T(M_1,M^T)
\]
is exact, where the last arrow is the boundary operator of the long exact sequence of the pair $(M_1,M^T)$. Thus, the image of $H^*_T(M)\to H^*_T(M^T)$ is the same as the image of the restriction map $H^*_T(M_1)\to H^*_T(M^T)$. 

GKM theory, named after Goresky, Kottwitz and MacPherson \cite{MR1489894}, now poses additional conditions on the action in order to simplify the structure of the one-skeleton of the action. We assume first of all that the action has only finitely many fixed points. Then at each fixed point $p\in M$ the isotropy representation decomposes into its weight spaces 
\[
T_pM = \bigoplus_{i=1}^n V_i,
\]
corresponding to weights $\alpha_i:\mft\to \R$ which are well-defined up to multiplication by $-1$. 
\begin{definition}We say that the action is GKM$_k$, $k\geq 2$, if it is equivariantly formal, has only finitely many fixed points, and at each fixed point $p$ any $k$ weights of the isotropy representation are linearly independent. 
\end{definition}

For $k=2$ one obtains the usual GKM conditions. If $M$ is a GKM$_k$-manifold with respect to some action of a torus $T$ and $T'\subset T$ is a subtorus of codimension $< k$, then the condition on the weights implies that every component of its fixed point set $M^{T'}$ is of dimension at most $2k$. In particular, for $T'$ of codimension one, each component is either a point or a two-dimensional $T$-manifold with fixed points. Hence, it can only be either $S^2$ or $\R P^2$, with $\R P^2$ occurring only if $M$ is non-orientable.

Consider first the case that $M$ is orientable, i.e., that the one-skeleton $M_1$ is a union of two-spheres. Then the GKM graph $\Gamma_M$ of the action is by definition the graph with one vertex for each fixed point, and one edge connecting two vertices for each two-sphere in $M_1$ containing the corresponding fixed points. We label the edge with the isotropy weight of the associated two-sphere.

In the non-orientable case, the graph encodes the possible presence of $\R P^2$'s in the one-skeleton via edges that connect a vertex corresponding to the unique fixed point contained in the $\R P^2$ with an auxiliary vertex representing the exceptional orbit; in \cite{goertsches13:_non_gkm} these vertices are drawn as stars. However, in our situation of a torus action these edges  have no impact on the equivariant cohomology at all, so we could as well leave them out from the graph without losing any cohomological information.

Summarizing, one obtains the following description of the equivariant cohomology algebra of an action of type GKM (\cite[Theorem 7.2]{MR1489894}; the non-orientable case is an easy generalization, see \cite[Section 3]{goertsches13:_non_gkm}):
\begin{theorem} Consider an action of a torus $T$ on a compact manifold $M$ of type GKM$_2$, with fixed points $p_1,\ldots,p_n$. Then, via the natural restriction map
\[
H^*_T(M)\longrightarrow H^*_T(M^T) = \bigoplus_{i=1}^n S(\mft^*),
\] 
the equivariant cohomology algebra $H^*_T(M)$ is isomorphic to the set of tuples $(f_i)$, with the property that if the vertices $i$ and $j$ in the associated GKM graph are joined by an edge with label $\alpha\in \mft^*$, then $f_i|_{\ker \alpha}=f_j|_{\ker \alpha}$.
\end{theorem}

If \(M\) is of class GKM$_3$, the \emph{two-skeleton} \(M_2=\{p\in M\mid \dim T\cdot p\leq 2\}\) of \(M\) is a union of four-dimensional $T$-invariant submanifolds. The $T$-action induces on each of these submanifolds an effective action of a two-dimensional torus. We call a subgraph of \(\Gamma_M\) which corresponds to the intersection of \(M_1\) with one of these four-dimensional manifolds a two-dimensional face of \(\Gamma_M\).
It is easy to see that for each pair \((e_1,e_2)\) of edges of \(\Gamma_M\) emanating from the same vertex \(v\) there is exactly one two-dimensional face of \(\Gamma_M\) which contains \(e_1\) and \(e_2\).

For a \(2m\)-dimensional GKM manifold \(M\) with fixed points \(p_1,\dots,p_n\) the restriction of the total equivariant (real) Pontryagin class \(p^T(M)\) of \(M\) to \(M^T\) is given by
\begin{equation}
\label{eq:7}
\sum_{i=1}^n\prod_{j=1}^m(1+\alpha_{ij}^2)\in \bigoplus_{i=1}^n S(\mft^*),
\end{equation}
where the \(\alpha_{ij}\) are the weights of the \(T\)-representation
\(T_{p_i}M\) at \(p_i\in M^T\)  (see for example \cite[Lemma
6.10]{kawakubo} and \cite[Corollary 15.5]{MR0440554}).
Since the restriction \(H^*_T(M)\rightarrow H^*_T(M^T)\) is injective, it follows that the equivariant Pontryagin classes of \(M\) are determined by the GKM graph of \(M\).

\section{Strategy of the proof}
\label{sec:main-results-7}
Let us indicate here the strategy of the proof of Theorem \ref{sec:main-results-8} and Corollary \ref{sec:main-results-9}. In Section \ref{sec:rankoneactions} we will determine all possible GKM graphs of torus actions on compact rank one symmetric spaces. Then in Section \ref{sec:main} we will show that under the given assumptions the GKM graph necessarily coincides with one of these, say of an action on a CROSS \(N\).
  Therefore by the GKM-description of equivariant cohomology we have an isomorphism of \(H^*(BT)\)-algebras \(H_T^*(M)\cong H_T^*(N)\).
  Since \(H^*(M)\cong H^*_T(M)/(H^{>0}(BT))\) and similarly for \(N\), it follows that \(H^*(M)\cong H^*(N)\).
The remark about the Pontryagin classes follows because \(p^T(M)\) is mapped to \(p(M)\) by the natural map \(H_T^*(M)\rightarrow H^*(M)\) and \(p^T(M)\) is determined by the GKM graph of \(M\) by (\ref{eq:7}).

\section{Torus actions on compact rank one symmetric spaces}
\label{sec:rankoneactions}
Let $M$ be an even-dimensional compact simply-connected symmetric space of rank one, i.e., either $S^{2n}$, $\C P^n$, $\HH P^n$ or $\OO P^2$. In the following we describe certain GKM actions on these manifolds and determine their GKM graphs, including their labeling with weights.

\subsection{The spheres}\label{subsec:spheres}

Let \(\alpha_i:T\rightarrow S^1\), \(i=1,\dots,n\) be characters of the torus \(T\).
Denote by \(V_{\alpha_i}\) the \(T\)-representation
\begin{align*}
  T\times \C &\rightarrow \C& (t,z)&\mapsto \alpha_i(t)z.
\end{align*}
Here and in the rest of Section \ref{sec:rankoneactions} we will denote the weight of this representation, i.e., the linear form $d\alpha_i\in {\mathfrak{t}}^*$, again by $\alpha_i$.

Assume that for any \(j_1,j_2\in\{1,\dots,n\}\), \(j_1\neq j_2\), \(\alpha_{j_1}\) and \(\alpha_{j_2}\) are linearly independent.
Then the restriction of the \(T\)-action to the unit sphere of \(V_{\alpha_1}\oplus\dots\oplus V_{\alpha_{n}}\oplus \R\) defines a GKM$_2$ action of \(T\) on \(S^{2n}\).
This action has \(2\) fixed points \(v_1,v_2\) corresponding to the two points in the intersection of the sphere with the \(\R\)-summand.
These two vertices are joined  in the GKM graph of \(S^{2n}\) by exactly \(n\) edges.
Moreover, the weights of these edges are given by the \(\pm \alpha_i\).

Thus, any labeling of $\Gamma_{S^{2n}}$ with pairwise linearly independent weights is realized by a GKM action on $S^{2n}$.

 \begin{figure}[htb]
 \includegraphics[width=150pt]{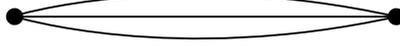}
 \caption{GKM graph of $S^6$}
 \label{figrp}
 \end{figure}

\subsection{The complex projective spaces}\label{subsec:CP^n}

Let \(\alpha_i:T\rightarrow S^1\), \(i=0,\dots,n\), be characters of the torus \(T\).
Denote by \(V_{\alpha_i}\) the \(T\)-representation
\begin{align*}
  T\times \C &\rightarrow \C& (t,z)&\mapsto \alpha_i(t)z.
\end{align*}
Assume that, if \(n>1\), for any pairwise distinct \(i,j_1,j_2\in\{0,\dots,n\}\),  \(\alpha_{j_1}-\alpha_i\) and \(\alpha_{j_2}-\alpha_i\) are linearly independent.
If \(n=1\) assume that \(\alpha_1\neq \alpha_0\).
Then the projectivization of \(V_{\alpha_0}\oplus\dots\oplus V_{\alpha_{n}}\) defines a GKM action of \(T\) on \(\C P^n\).
This action has \(n+1\) fixed points \(v_0,\dots,v_{n}\) corresponding to the weight spaces of the above \(T\)-representation.
The GKM graph of this action is a complete graph on these fixed points.
Moreover, the weight of the edge from \(v_i\) to \(v_j\) is given by \(\pm(\alpha_i-\alpha_j)\).

 \begin{figure}[htb]
 \includegraphics[width=150pt]{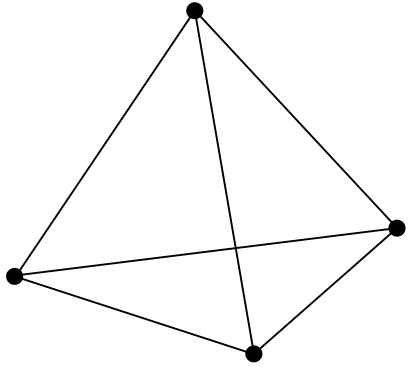}
 \caption{GKM graph of $\C P^3$}
 \label{figrp}
 \end{figure}
In particular, in this type of example an arbitrary set of non-zero pairwise linearly independent weights \(\gamma_1,\dots,\gamma_n\), such that, if for any pairwise distinct \(i,j_1,j_2\in\{1,\dots,n\}\),  \(\gamma_{j_1}-\gamma_i\) and \(\gamma_{j_2}-\gamma_i\) are linearly independent,    can occur as labels of the edges adjacent to $v_0$.

\subsection{The hyperbolic projective spaces}

Let \(\alpha_i:T\rightarrow S^1\), \(i=0,\dots,n\), be characters of the torus \(T\).
Denote by \(V_{\alpha_i}\) the quaternionic \(T\)-representation
\begin{align*}
  T\times \mathbb{H} &\rightarrow \mathbb{H}& (t,q)&\mapsto \alpha_i(t)q.
\end{align*}
Assume that, if \(n>1\), for any pairwise distinct \(i,j_1,j_2\in\{0,\dots,n\}\),  \(\alpha_{j_1}\pm\alpha_i\) and \(\alpha_{j_2}\pm\alpha_i\) are linearly independent.
If \(n=1\) assume that \(\alpha_0+\alpha_1\) and \(\alpha_0-\alpha_1\) are linearly independent.
Because the action of \(T\) on \(V_{\alpha_0}\oplus\dots\oplus V_{\alpha_{n}}=\mathbb{H}^{n+1}\) commutes with the action of \(\mathbb{H}^*\) given by multiplication from the right, we get a GKM action of \(T\) on \(\mathbb{H} P^n=(\mathbb{H}^{n+1}\setminus\{0\})/\mathbb{H}^*\).
This action has \(n+1\) fixed points \(v_0,\dots,v_{n}\) corresponding to the weight spaces of the above \(T\)-representation.
In the GKM graph of this action any two fixed points are joined by exactly two edges.
Moreover, the weights of the edges from \(v_i\) to \(v_j\) are given by \(\pm(\alpha_i-\alpha_j)\) and \(\pm(\alpha_i+\alpha_j)\).

 \begin{figure}[htb]
 \includegraphics[width=150pt]{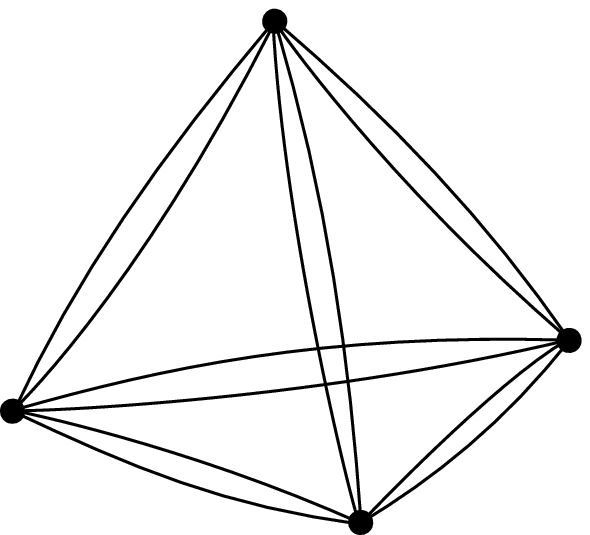}
 \caption{GKM graph of $\HH P^3$}
 \label{figrp}
 \end{figure}
 Note that in this case, the set of weights at the edges adjacent to, say, $v_0$, is not arbitrary. Indeed, the weights corresponding to the two edges connecting $v_0$ with $v_i$ add up (after choosing the right signs) to $2\alpha_0$, independent of $i$.

\subsection{The Cayley plane}

The Cayley plane $\OO P^2$ can be defined as the homogeneous space $F_4/\Spin(9)$. This is a homogeneous space of the form $G/K$ satisfying $\rank G = \rank K$. For a homogeneous space $G/K$ of this type, let $T\subset K$ be a maximal torus and consider the $T$-action on $G/K$ by left multiplication. It is known that this action is equivariantly formal and is of type GKM$_2$ \cite{GuilleminHolmZara}. The GKM graph of an action of this type can be described completely in terms of the root systems $\Delta_G$ and $\Delta_K$ of $G$ and $K$, see \cite[Theorem 2.4]{GuilleminHolmZara}: first of all, the set of fixed points of the action is given by the quotient of Weyl groups $W_G/W_K$. At the origin $eK$, the tangent space of $G/K$ is $K$-equivariantly isomorphic to the quotient $\mfg/\mfk$, which implies that the weights of the isotropy representation at $eK$ are given by those roots of $G$ which are not roots of $K$. Moreover, two vertices corresponding to elements $wW_K$ and $w'
 W_K$ of $W_G/W_K$ are on a common edge if and only if $wW_K =\sigma_\alpha  w' W_K$ for some $\alpha\in \Delta_G\setminus \Delta_K$, where $\sigma_\alpha$ denotes the reflection at $\alpha$.

\begin{remark} Of course, also the other compact rank one symmetric spaces are homogeneous spaces of this type, and we could have used the results of \cite{GuilleminHolmZara} to describe the GKM graphs of the torus actions on these spaces as well.
\end{remark}

Because $\dim H^*(\OO P^2) = 3$, the action of a maximal torus $T\subset \Spin(9)$ has exactly three fixed points, which correspond to three vertices in the associated GKM graph. Since $\dim \OO P^2=16$, from each vertex emerge eight edges. Because of the $W_{F_4}$-action on the graph which is transitive on the vertices, it follows that any two vertices are connected by four edges.

 \begin{figure}[htb]
 \includegraphics[width=150pt]{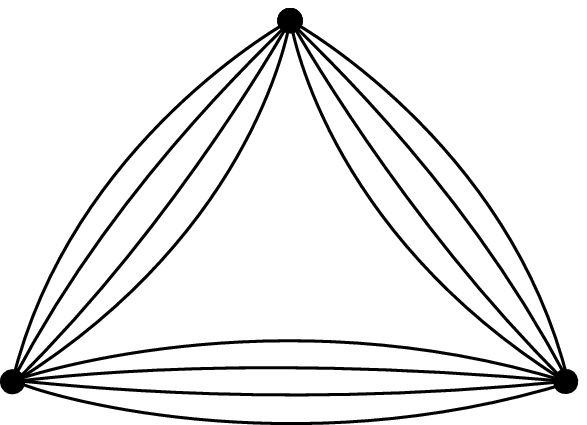}
 \caption{GKM graph of $\OO P^2$}
 \label{figrp}
 \end{figure}

The root system of $F_4$ is described explicitly in \cite[Proposition 2.87]{Knapp}: Consider $\R^4$, equipped with the standard inner product and standard basis \(e_1,\dots,e_4\),
\[
\Delta_{F_4}= \begin{cases} \pm e_i \\ \pm e_i\pm e_j & \text{ for }i\neq j\\ \frac12 (\pm e_1\pm e_2\pm e_3\pm e_4).\end{cases}
\]
The root system of $\Spin(9)$ is contained in $\Delta_{F_4}$ as the subroot system
\[
\Delta_{\Spin(9)} = \begin{cases}\pm e_i \\ \pm e_i\pm e_j & \text{ for }i\neq j,\end{cases}
\]
see \cite[p.\ 55]{Adams}.

Now we can determine the labeling of the GKM graph of \(\OO P^2\) from the fact that the labeling of the graph is invariant under the action of \(W_{F_4}\) on the graph. 

More precisely, if $\alpha$ and $\beta$ are weights in $\Delta_{F_4}\setminus \Delta_{\Spin(9)}$, then the reflection $\sigma_\alpha\beta$ of $\beta$ at $\alpha$ is $\pm \beta$ if and only if the number of minus signs in $\alpha$ and $\beta$ coincide modulo two, and if the number of signs is not congruent modulo two, then $\sigma_\alpha\beta$ is one of the roots $\pm e_i$. (See also the proof of Proposition 2.87 in \cite{Knapp} for a description of the action of the Weyl group of $F_4$ on $\Delta_{F_4}$.) This implies that if \(v_0,v_1,v_2\) are the vertices of \(\Gamma_{\OO P^2}\), then the weights at the edges between two of them, say \(v_0\) and \(v_1\), are given (up to sign) by
\begin{align*}
  \frac{1}{2}(-e_1 -e_i+ \sum_{j\neq i,1}e_j)\quad (i=2,3,4)\quad  \text{ and } \quad  \frac{1}{2}\sum_{j=1}^4e_j,
\end{align*}
the weights at the edges between \(v_0\) and \(v_2\) are given (up to sign) by
\begin{align*}
   \frac{1}{2}(-e_i+ \sum_{j\neq i}e_j)\qquad (i=1,\dots,4),
\end{align*}
and the weights at the edges between \(v_1\) and \(v_2\) are given by \(\pm e_1,\dots,\pm e_4\).

\begin{remark} The examples given in this section show that every simply connected compact symmetric space of rank one admits a torus action of type GKM$_3$. Note that in order to check whether an action of a torus $T\subset K$ on a homogeneous space $G/K$ with $\rank G = \rank K$ is GKM$_3$, we only need to check the $3$-independence of the weights in one of the fixed points: the $3$-independence in the other fixed points then follows from the $W(G)$-action on the GKM graph which is transitive on the vertices.

Moreover, one sees that there are actions on these spaces for which all weights appearing in the GKM graph of the action are primitive vectors in the weight lattice of \(T\). To see this just take the \(\alpha_i\) in the construction of the actions on \(S^{2n}\), \(\C P^n\) and \(\HH P^n\) to be primitive vectors and pairwise linear independent. For the action on \(\OO P^2\) one sees that the weights are primitive vectors in the weight lattice of a maximal torus of \(F_4\), because this lattice is generated by the roots of \(F_4\).
\end{remark}

\section{The GKM graphs of positively curved manifolds}
\label{sec:main}

In this section we determine the GKM graphs of positively curved GKM$_3$- and GKM$_4$-manifolds.
The key observation is the following lemma.

\begin{lemma}
\label{sec:main-results}
  Let \(M\) be an orientable GKM$_3$-manifold with an invariant metric of positive sectional curvature.
  Then all two dimensional faces of the GKM graph \(\Gamma_M\) of \(M\) have two or three vertices.
\end{lemma}
\begin{proof}
  A two-dimensional face of \(\Gamma_M\) is the GKM graph of a four-dimensional invariant submanifold \(N\) of \(M\) on which a two-dimensional torus acts effectively.
  \(N\) is a fixed point component of the action of some subtorus \(T'\subset T\).
  Therefore it is totally geodesic in \(M\) and the induced metric has positive sectional curvature.
  The classification results of four-dimensional \(T^2\)-manifolds with positive sectional curvature given in \cite{MR1255926} imply that \(N\) is diffeomorphic to \(S^4\) or \(\C P^2\). As the Euler characteristic of these manifolds is at most $3$, the claim follows.
\end{proof}

Let \(M\) be a positively curved GKM$_3$-manifold and \(N\subset M\) a four-dimensional invariant submanifold corresponding to a two-dimensional face of \(\Gamma_M\) which is a triangle.
Then \(N\) is equivariantly diffeomorphic to \(\C P^2\) equipped with one of the actions described in Section~\ref{subsec:CP^n}
\cite[Theorem 4.3]{orlik-raymond}, \cite[Theorem 2]{melvin}.
Let \(\alpha\) and \(\beta\) denote the weights at two of the three edges in \(\Gamma_N\).
Note that these weights are determined up to sign.
Then the weight at the third edge of \(\Gamma_N=\Gamma_{\C P^2}\) is given by \(\pm \alpha\pm \beta\).

\begin{lemma}
\label{sec:main-results-1}
  An orientable GKM$_4$-manifold with an invariant metric of positive sectional curvature has the same GKM graph as a torus action on \(S^{2n}\) or \(\C P^n\).
\end{lemma}
\begin{proof} For three distinct vertices $v_1,v_2,v_3$ in $\Gamma_M$ we denote by $K_{ij}$ the set of edges between $v_i$ and $v_j$. Then each pair \((e_1,e_2)\in K_{12}\times K_{13}\) spans a two-dimensional face of \(\Gamma_M\). By Lemma \ref{sec:main-results}, the vertices of this face are exactly $v_1$, $v_2$ and $v_3$. We denote the unique edge contained in this face which connects $v_2$ and $v_3$ by $\phi(e_1,e_2)$; in this way we obtain a map $\phi:K_{12}\times K_{13}\to K_{23}$. 

  If \(\alpha_1,\alpha_2\) are the weights at \(e_1\) and \(e_2\), respectively, then the weight at \(\phi(e_1,e_2)\) is given by \(\pm \alpha_1\pm \alpha_2\).
  Therefore if \((e_1',e_2')\in K_{12}\times K_{13}\) is another pair with \(\phi(e_1',e_2')=\phi(e_1,e_2)\), then we get the relation:
  \begin{equation*}
    \pm \alpha_1\pm \alpha_2=\pm \alpha_1'\pm \alpha_2'.
  \end{equation*}
  This contradicts the $4$-independence of the weights at \(v_1\). We have thus shown that $\phi$ is injective, i.e., \(\# K_{12}\cdot \# K_{13} \leq \# K_{23}\). But this relation holds also for any permutation of $v_1,v_2$ and $v_3$, hence if one \(\# K_{ij}>1\), the other two \(K_{ij}\) must be empty. It follows that no edge of $\Gamma_M$ is contained at the same time in a two-dimensional face with three vertices and another two-dimensional face with two vertices.

  Now the GKM-graph of a GKM$_3$-manifold, all of whose two-dimensional
  faces are biangles, is $\Gamma_{S^{2n}}$, the graph with two vertices and \(n\) edges.
  The GKM-graph of a GKM$_3$-manifold, all of whose two-dimensional faces are triangles, is the complete graph $\Gamma_{\C P^{n}}$ on \(n+1\) vertices.

As argued in Subsection \ref{subsec:spheres}, any labeling of $\Gamma_{S^{2n}}$ with pairwise linearly independent weights is realized by a torus action on $S^{2n}$. For $\Gamma_{\C P^n}$, the statement that the weights are those of a torus action on $\C P^n$ will be shown in Lemma \ref{lem:weightsCP^n} below (for later use already for GKM$_3$-actions).
\end{proof}

\begin{lemma}\label{lem:weightsCP^n} Consider a GKM$_3$-manifold $M$ with the GKM graph $\Gamma_{\C P^n}$. Then the induced labeling of the GKM graph is the same as that of a torus action on $\C P^n$.
\end{lemma} 
\begin{proof} Let $v_0,\ldots,v_n$ be the vertices of the GKM graph, and $\gamma_{ij}$ the weight of the edge between $v_i$ and $v_j$ (which is determined up to sign). By the considerations in Subsection \ref{subsec:CP^n} we only have to show that we can choose the signs in such a way that $\gamma_{ij}=\gamma_{0i}-\gamma_{0j}$. We can assume without loss of generality that
\[
\gamma_{1i} = \gamma_{01} - \gamma_{0i}
\]
for all $i$, and that
\[
\gamma_{ij} = \gamma_{0i} \pm \gamma_{0j}.
\]
This implies
\[
\gamma_{ij} = \gamma_{0i}\pm \gamma_{0j} = \gamma_{01} - \gamma_{1i} \pm (\gamma_{01} - \gamma_{1j}),
\]
but on the other hand $\gamma_{ij}$ is also a linear combination of $\gamma_{1i}$ and $\gamma_{1j}$. Hence, the $3$-independence shows that $\gamma_{ij} = \gamma_{0i}-\gamma_{0j}$.
\end{proof} 

As explained in Section~\ref{sec:main-results-7}, the first part of Theorem~\ref{sec:main-results-8} follows from Lemma \ref{sec:main-results-1}. The second part will follow from the following sequence of lemmas.

\begin{lemma}
\label{sec:main-results-2}
  Let \(M\) be an orientable GKM$_3$-manifold with an invariant metric of positive sectional curvature.
  Then \(\Gamma_M\) is equal to a graph which is constructed from a simplex by replacing each edge by $k$ edges, \(k\in \N\).
\end{lemma}
\begin{remark}
Note that all the GKM graphs listed in Section \ref{sec:rankoneactions} arise in this way. The graph of $S^{2n}$, for example, is obtained from the $1$-simplex.
\end{remark}
\begin{proof}
  Assume that \(\Gamma_M\) has at least three vertices, i.e., is not a single point and not \(\Gamma_{S^{2n}}\).
  Then the statement of the lemma follows as soon as we can show that for every choice of vertices \(v_1,v_2,v_3\) in a two-dimensional face of \(\Gamma_M\) we have \(\# K_{12}=\# K_{13}=\# K_{23}\).

  But for that it is sufficient to show that \(\# K_{12}\leq \# K_{23}\), by the symmetry of the statement. For \(f\in K_{13}\), consider the map \(\psi: K_{12}\rightarrow K_{23}\), \(e_2\mapsto \phi(e_2,f)\), where $\phi:K_{12}\times K_{13}\to K_{23}$ is the map defined in Lemma \ref{sec:main-results-1}, sending two edges to the unique third edge in the two-dimensional face spanned by them.  
 Now if \(\psi(e_1)= \psi(e_2)\), then \((\psi(e_1),f)\) and \((\psi(e_2),f)\) span the same triangle in \(\Gamma_M\).
  Therefore we must have \(e_1=e_2\) in this case, i.e., $\psi$ is injective, which shows that \(\# K_{12}\leq \# K_{23}\).
\end{proof}

The following lemma implies Corollary~\ref{sec:main-results-9}.

\begin{lemma}
\label{sec:main-results-3}
  A GKM$_3$-manifold with an invariant metric of positive sectional curvature and an invariant almost complex structure has the same GKM graph as \(\C P^n\).
\end{lemma}
\begin{proof}
  If there is an almost complex structure on a GKM-manifold \(N\), one can modify the definition of the GKM graph so that it contains information about the complex structure (see for example \cite[Section 1]{MR1919427}).
  In this case the weights of the \(T\)-representations \(T_xN\), \(x\in N^T\), have a preferred sign.
  We label the oriented edges emanating from \(x\in \Gamma_N\) by the weights of the representation \(T_xM\).
  For an oriented edge \(e\) denote by \(\alpha(e)\) the weight at \(e\).
  Then we have \(\alpha(\bar{e})=-\alpha(e)\), where \(\bar{e}\) denotes \(e\) equipped with the inverse orientation.

  Now assume that \(N\) is four-dimensional and let \(v_1,v_2\) be two vertices in \(\Gamma_N\) joined by an edge \(e\), denote by \(e_1\) and \(e_2\) the other oriented edges emanating from \(v_1\) and \(v_2\), then we must have \(\alpha(e_1)=\alpha(e_2)\mod \alpha(e)\).

  From these two properties we get a contradiction if we assume that \(\Gamma_N\) has only two vertices and two edges.

  Therefore there is no biangle in \(\Gamma_M\).
  Now, by Lemma~\ref{sec:main-results-2}, the claim follows.
\end{proof}

\begin{remark}
  As pointed out by one of the referees, one can also prove the above
  lemma as follows.
  If \(M\) admits an invariant almost complex structure, then all
  components of \(M^{T'}\) for any subtorus \(T'\subset T\) admit almost
  complex structures.
  A two-dimensional face \(F\) of \(\Gamma_M\) is the GKM graph of such a
  fixed point component \(N\).
  \(F\) is a biangle if and only if \(N\) is diffeomorphic to
  \(S^4\).
  But \(S^4\) does not admit any almost complex structure.
  Hence, there are no biangles in \(\Gamma_M\).
  Therefore the lemma follows as in the above proof.
\end{remark}

\begin{lemma}
\label{sec:main-results-4}
  Let \(M\) be an orientable GKM$_3$-manifold with an invariant metric of positive sectional curvature.
  Then \(\Gamma_M\) is equal to \(\Gamma_{S^{2n}}\) or a graph which is constructed from a simplex by replacing each edge by \(k\) edges, \(k=1,2,4\).
\end{lemma}
\begin{proof}
  Assume that \(\Gamma_M\) is not \(\Gamma_{S^{2n}}\).
  Then there is a triangle  \(v_1,v_2,v_3\) in \(\Gamma_M\).
  Denote by \(\alpha_i\) the weights of the edges between \(v_1\) and \(v_2\),
  by \(\beta_i\) the weights of the edges between \(v_1\) and \(v_3\),
  by \(\gamma_i\) the weights of the edges between \(v_2\) and \(v_3\).

  Since the weights are only determined up to sign we may assume that
  \begin{equation}
\label{eq:1}
    \gamma_i=\alpha_1-\beta_i=\epsilon_{ij} \alpha_j +\delta_{ij} \beta_{\sigma_j(i)}
  \end{equation}
  with \(\epsilon_{ij},\delta_{ij}\in \{\pm 1\}\) and permutations \(\sigma_j\), where \(\sigma_1=\id\).
    
\begin{sublemma}
\label{sec:gkm-graphs-posit}
  For $i\neq j$ the permutations \(\sigma_j^{-1}\circ \sigma_i\) are all of order two. Moreover, they do not have fixed points; in other words, for each $l$, the map $j\mapsto \sigma_j(l)$ is a permutation as well.
\end{sublemma}
\begin{proof}[{Proof of Sublemma \ref{sec:gkm-graphs-posit}}]
Fix $i\neq j$; we consider the permutation $\sigma_j^{-1}\circ \sigma_i$ and want to show that it is of order $2$ without fixed points. 

To show that it has no fixed points assume $\sigma_i(l) = \sigma_j(l)$ for some $l$. Then  \eqref{eq:1} applied twice gives 
\[
\epsilon_{li}\alpha_i + \delta_{li}\beta_{\sigma_i(l)} = \gamma_l = \epsilon_{lj}\alpha_j + \delta_{lj} \beta_{\sigma_j(l)},
\] a contradiction to the $3$-independence.
Therefore there are no fixed points. 

Fix a number $A$, and let $B=\sigma_j^{-1}\circ\sigma_i(A)$, i.e. \(\sigma_i(A) = \sigma_j(B)\).
We have
\begin{equation}
\label{eq:9}
\gamma_A = \epsilon_{Ai}\alpha_i + \delta_{Ai}\beta_{\sigma_i(A)} = \epsilon_{Aj}\alpha_j + \delta_{Aj}\beta_{\sigma_j(A)}
\end{equation}
and
\begin{equation}
\gamma_B = \epsilon_{Bi}\alpha_i + \delta_{Bi}\beta_{\sigma_i(B)}=\epsilon_{Bj}\alpha_j + \delta_{Bj}\beta_{\sigma_j(B)}.
\end{equation}
Then the term $\epsilon_{Bi}\gamma_A - \epsilon_{Ai}\gamma_B$ is equal to the following two expressions:
\begin{equation}\label{eq:t:11}
\begin{split}
\epsilon_{Bi}\gamma_A - \epsilon_{Ai}\gamma_B&=\epsilon_{Bi} \delta_{Ai}\beta_{\sigma_i(A)} -\epsilon_{Ai}\delta_{Bi}\beta_{\sigma_i(B)}\\ &= (\epsilon_{Bi}\epsilon_{Aj}-\epsilon_{Ai}\epsilon_{Bj})\alpha_j + \epsilon_{Bi}\delta_{Aj}\beta_{\sigma_j(A)}-\epsilon_{Ai}\delta_{Bj}\beta_{\sigma_j(B)}.
\end{split}
\end{equation}

Assume that
\(\epsilon_{Bi}\epsilon_{Aj}=-\epsilon_{Ai}\epsilon_{Bj}\).
Then $\epsilon_{Bi}\gamma_A+ \epsilon_{Ai}\gamma_B$ is equal to the following:
\begin{equation}\label{eq:q:11}
\epsilon_{Bi}\gamma_A+ \epsilon_{Ai}\gamma_B = \epsilon_{Bi}\delta_{Aj}\beta_{\sigma_j(A)}+\epsilon_{Ai}\delta_{Bj}\beta_{\sigma_j(B)}.
\end{equation}

By adding and subtracting equations (\ref{eq:q:11}) and (\ref{eq:t:11}) we get
\begin{equation}
  \label{eq:q:12}
  2\epsilon_{Bi}\gamma_A= \epsilon_{Bi}\delta_{Aj}\beta_{\sigma_j(A)}+\epsilon_{Ai}\delta_{Bj}\beta_{\sigma_j(B)}+\epsilon_{Bi} \delta_{Ai}\beta_{\sigma_i(A)} -\epsilon_{Ai}\delta_{Bi}\beta_{\sigma_i(B)}
\end{equation}
and
\begin{equation}
  \label{eq:q:13}
  2\epsilon_{Ai}\gamma_B= \epsilon_{Bi}\delta_{Aj}\beta_{\sigma_j(A)}+\epsilon_{Ai}\delta_{Bj}\beta_{\sigma_j(B)}-\epsilon_{Bi} \delta_{Ai}\beta_{\sigma_i(A)} +\epsilon_{Ai}\delta_{Bi}\beta_{\sigma_i(B)}.
\end{equation}

Since \(\beta_{\sigma_i(A)}=\beta_{\sigma_j(B)}\), \(\gamma_A\) or \(\gamma_B\) is a linear combination of \(\beta_{\sigma_j(A)}\) and \(\beta_{\sigma_i(B)}\). This gives a contradiction to the \(3\)-independence of the weights.
Hence our assumption is wrong; and we have \(\epsilon_{Bi}\epsilon_{Aj}=\epsilon_{Ai}\epsilon_{Bj}\).
Therefore it follows from equation (\ref{eq:t:11}) and the \(3\)-independence that \(\sigma_j(A)=\sigma_i(B)\).
This shows that \(\sigma_j^{-1}\circ \sigma_i\) has order two.
\end{proof}

 In particular, the sublemma shows that all the \(\sigma_j=\sigma_1^{-1}\circ \sigma_j\), for $j\neq 1$, are of order two. Again by the sublemma, they also commute because
 \[
 \sigma_j\circ \sigma_i = \sigma_j^{-1}\circ \sigma_i = (\sigma_j^{-1}\circ \sigma_i)^{-1} = \sigma_i^{-1}\circ \sigma_j = \sigma_i \circ \sigma_j.
 \]
  Let \(G\) be the subgroup of the permutation group generated by the \(\sigma_j\).
  Then we have an epimorphism \(\mathbb{Z}_2^{k-1}\rightarrow G\).
  Since the \(\sigma_j^{-1}\circ \sigma_i\) do not have fixed points, \(G\) acts transitively on \(\{1,\dots,k\}\).
  Therefore we have \(\{1,\dots,k\}\cong \mathbb{Z}_2^{k-1}/H\) for a subgroup \(H\) of \(\mathbb{Z}_2^{k-1}\). In particular \(k\) is a power of two.

Let \(i\in \{1,\dots,k\}\).

\begin{sublemma}
\label{sec:gkm-graphs-posit-1}
   For \(j>1\) we have $\delta_{ij}=1$ and $\epsilon_{\sigma_j(i)j}=\epsilon_{ij}$, i.e., 
\begin{align*}
  \gamma_i&=\epsilon_{ij} \alpha_j + \beta_{\sigma_j(i)}& \gamma_{\sigma_j(i)}&=\epsilon_{ij} \alpha_j + \beta_{i}.
\end{align*}
\end{sublemma}

\begin{proof}[{Proof of Sublemma \ref{sec:gkm-graphs-posit-1}}]
  From the above relation (\ref{eq:1}) it follows that
\begin{align*}
  \alpha_1&=\gamma_i + \beta_i=\epsilon_{ij} \alpha_j +\delta_{ij} \beta_{\sigma_j(i)} + \beta_i\\
  \alpha_1&=\gamma_{\sigma_j(i)} + \beta_{\sigma_j(i)}=\epsilon_{\sigma_j(i)j} \alpha_j +\delta_{\sigma_j(i)j} \beta_{i} + \beta_{\sigma_j(i)}.
\end{align*}
Now the three-independence of the weights implies that we have $\delta_{ij}=1$ and \(\epsilon_{\sigma_j(i)j}=\epsilon_{ij}\).
\end{proof}

\begin{sublemma}
\label{sec:gkm-graphs-posit-2} For $j>2$ we have
$\epsilon_{\sigma_j(i)2}=-\epsilon_{i2}$.
\end{sublemma}

\begin{proof}[{Proof of Sublemma \ref{sec:gkm-graphs-posit-2}}]
  From Sublemma \ref{sec:gkm-graphs-posit-1} it follows that
\begin{align*}
  \epsilon_{i2}\alpha_2&=\gamma_{\sigma_2(i)} - \beta_i=\epsilon_{\sigma_2(i)j} \alpha_j + \beta_{\sigma_j\circ\sigma_2(i)} - \beta_i\\
  \epsilon_{\sigma_j(i)2} \alpha_2&=\gamma_{\sigma_j(i)} - \beta_{\sigma_j\circ\sigma_2(i)}=\epsilon_{ij} \alpha_j + \beta_{i} - \beta_{\sigma_j\circ\sigma_2(i)}.
\end{align*}
Now the three-independence of the weights implies that \(\epsilon_{i2}=-\epsilon_{\sigma_j(i)2}\).
\end{proof}

Because we know that $k$ is a power of two, if $k>2$, then $k\geq 4$, so we may choose \(j>j'>2\). 
Then, by applying Sublemma~\ref{sec:gkm-graphs-posit-2} twice, we have
\[
\epsilon_{\sigma_j(\sigma_{j'}(i))2}=-\epsilon_{\sigma_{j'}(i)2}=\epsilon_{i2}.
\]
Since, by Sublemma \ref{sec:gkm-graphs-posit}, there is a \(j''\) such that \(\sigma_{j''}(i)=\sigma_j\circ\sigma_{j'}(i)\) and $\sigma_j\circ\sigma_{j'}$ does not have fixed points, it follows, again by Sublemma \ref{sec:gkm-graphs-posit-2}, that $j''=2$.  Because this holds for each \(i\), it follows that \(\sigma_{2}=\sigma_j\circ\sigma_{j'}\).
 Hence it follows that \(k\leq 4\).
 This proves the lemma.
\end{proof}

\begin{lemma}
\label{sec:main-results-5}
  If we have \(k=4\) in the situation of Lemma~\ref{sec:main-results-4}, then \(\Gamma_M\) is a triangle with each edge replaced by four edges.
\end{lemma}
\begin{proof}
  Assume that \(\Gamma_M\) is a higher dimensional simplex.
  Then there are four vertices \(v_0,\dots,v_3\) in \(\Gamma_M\).
  Denote by \(\alpha_i\) the weights at the edges between \(v_0\) and \(v_1\),
  by \(\beta_i\) the weights at the edges between \(v_0\) and \(v_2\),
  by \(\gamma_i\) the weights at the edges between \(v_0\) and \(v_3\).

  By Sublemma \ref{sec:gkm-graphs-posit-1} we may assume that the weights at the edges between \(v_1\) and \(v_2\)
  are given by
  \begin{equation}\label{eq:new1}
    \alpha_1-\beta_i=\pm \alpha_2 + \beta_{\sigma(i)}
  \end{equation}
and that the weights at the edges between \(v_1\) and \(v_3\)
  are given by
  \begin{equation}\label{eq:new2}
    \alpha_1-\gamma_i=\epsilon_i \alpha_2 + \gamma_{\sigma'(i)},
  \end{equation}
for some permutations \(\sigma,\sigma'\) and \(\epsilon_i\in \{\pm1\}\).

At first we show that we may assume that the weights at the eges between \(v_2\) and \(v_3\) are given by \(\beta_1-\gamma_i\).
Since these weights are determined only up to sign we may assume that there are \(c_j\in \{\pm 1\}\) such that these weights are given by
\begin{equation*}
  \beta_1+c_j\gamma_j.
\end{equation*}
Hence, for each \(j=1,\dots,4\) there are \(a,b\in \{\pm 1\}\) and \(i\in \{1,\dots,4\}\) such that
\begin{equation}
  \label{eq:6}
  \alpha_1-\beta_1=a(\alpha_1-\gamma_i)+b(\beta_1+c_j\gamma_j).
\end{equation}

At first assume that \(i\neq j\).
By 3-independence at \(v_0\) we have \(a=-1\) and \(b=1\).
Then we have
\begin{equation*}
  (\alpha_1-\gamma_j)-(\beta_1-\gamma_j)=-(\alpha_1-\gamma_i)+(\beta_1+c_j\gamma_j).
\end{equation*}
We get a contradiction to 3-independence at \(v_3\) if \(c_j=-1\). Hence, we must have \(c_j=1\). By Sublemma~\ref{sec:gkm-graphs-posit-1}
there is a \(k\in \{1,\dots,4\}\) such that
\begin{equation*}
  \alpha_1-\gamma_i=\pm\alpha_k+\gamma_j.
\end{equation*}
Therefore we get
\begin{equation*}
  2\alpha_1-2\beta_1=\gamma_i+\gamma_j=\alpha_1\mp\alpha_k.
\end{equation*}
This is a contradiction to the 3-independence at \(v_0\).

Therefore we have \(i=j\). In this case it follows from 3-independence at \(v_0\) and Equation~(\ref{eq:6}) that \(a=1\) and \(b=c_j=-1\).
Therefore, by Sublemma~\ref{sec:gkm-graphs-posit-1}, the weights at the edges between \(v_2\) and \(v_3\)
  are given by
\begin{equation} \label{eq:new3}
    \beta_1-\gamma_i=\delta_i \beta_2 + \gamma_{\sigma''(i)},
  \end{equation}
where \(\sigma''\) is a permutation and \(\delta_i\in \{\pm1\}\). Note that by Sublemma \ref{sec:gkm-graphs-posit-2} not all $\delta_i$ are equal.

We may assume that \(\sigma(1)=2\). Then we have, using Equations \eqref{eq:new1}, \eqref{eq:6}, \eqref{eq:new2} and \eqref{eq:new3}, that
\begin{align*}
  \pm \alpha_2+\beta_2=\alpha_1-\beta_1&=(\alpha_1-\gamma_i)-(\beta_1-\gamma_i)\\
&=(\epsilon_i\alpha_2+ \gamma_{\sigma'(i)})-(\delta_i \beta_2 + \gamma_{\sigma''(i)}).
\end{align*}
Let \(m\) be the order of \(\sigma'^{-1}\sigma''\). Then we have:
\begin{multline*}
  m(\pm\alpha_2+\beta_2)=\sum_{l=0}^{m-1}((\epsilon_{(\sigma'^{-1}\sigma'')^l(i)}\alpha_2+\gamma_{\sigma'(\sigma'^{-1}\sigma'')^l(i)})\\ -(\delta_{(\sigma'^{-1}\sigma'')^l(i)}\beta_2+\gamma_{\sigma''(\sigma'^{-1}\sigma'')^l(i)}))\\
  =\sum_{l=0}^{m-1}(\epsilon_{(\sigma'^{-1}\sigma'')^l(i)}\alpha_2-\delta_{(\sigma'^{-1}\sigma'')^l(i)}\beta_2).
\end{multline*}
Therefore it follows that \(\delta_i=-1\) for all \(i\), a contradiction to the fact that not all $\delta_i$ are equal.
\end{proof}

We have shown that the GKM graph of an isometric GKM$_3$-action on a positively curved manifold is either $\Gamma_{S^{2n}}$, $\Gamma_{\C P^n}$, $\Gamma_{\HH P^n}$ or $\Gamma_{\OO P^2}$. To complete the proof of Theorem \ref{sec:main-results-8} we need to show that the induced labeling of the GKM graph is one of those described in Section \ref{sec:rankoneactions}. For $\Gamma_{S^{2n}}$ there is nothing to show, and for $\Gamma_{\C P^n}$ we already proved this in Lemma \ref{lem:weightsCP^n}. Let us consider now $\Gamma_{\HH P^n}$.

\begin{lemma}
\label{sec:main-results-10} Consider an orientable GKM$_3$-manifold $M$ with GKM graph $\Gamma_{\HH P^n}$. Then the induced labeling of the GKM graph is the same as that of a torus action on \(\mathbb{H} P^n\). 
\end{lemma}
\begin{proof}
Let \(v_0,v_1,\dots,v_m\) be the vertices of \(\Gamma_M\), and denote by $\gamma_{ij},\gamma_{ij}'$ the weights at the edges between $v_i$ and $v_j$. These weights are well-defined up to sign, and up to permutation of $\gamma_{ij}$ and $\gamma_{ij}'$. We can choose $\gamma_{0k}, \gamma_{1k}, \gamma_{1k}'$ and $\gamma_{01}'$ in such a way that for all $j>1$ we have (with Sublemma \ref{sec:gkm-graphs-posit-1} in mind)
\begin{align*}
\gamma_{1j}&=\gamma_{0j}-\gamma_{01} = \epsilon_{j} \gamma_{0j}'+\gamma_{01}' \\
\gamma_{1j}'&=\gamma_{0j}-\gamma_{01}' = \epsilon_{j} \gamma_{0j}' + \gamma_{01}
\end{align*}
for some signs $\epsilon_{j}$. Then we fix the sign of $\gamma_{0j}'$, $j>1$, such that $\epsilon_j=-1$. It follows that $\gamma_{01}+\gamma_{01}' = \gamma_{0j} + \gamma_{0j}'$ for all $j>1$, hence
\begin{equation}\label{eq:010j}
\gamma_{0i}+\gamma_{0i}' = \gamma_{0j} + \gamma_{0j}'
\end{equation}
for all $i,j>1$. We define
\[
\alpha_0 = \frac12(\gamma_{01} + \gamma_{01}'),\qquad \alpha_j = \frac12(\gamma_{0j}-\gamma_{0j}')\quad (j>1)
\]
so that
\[
\gamma_{0j} = \alpha_0 + \alpha_j\quad \text{and} \quad \gamma_{0j}' = \alpha_0-\alpha_j.
\]
We have to show that we can choose $\gamma_{ij}$ and $\gamma_{ij}'$ ($i,j>1$) in such a way that $\gamma_{ij} = \alpha_i + \alpha_j$ and $\gamma_{ij}' = \alpha_i - \alpha_j$.

For that we choose them such that
\begin{align}
\gamma_{ij} &= \gamma_{0j} + \eta_{ij} \gamma_{0i}' = \pm \gamma_{0j}' \pm \gamma_{0i} \label{eq:eta1}\\
\gamma_{ij}' &= -\gamma_{0j} + \eta_{ij}' \gamma_{0i} = \pm \gamma_{0j}' \pm \gamma_{0i}' \label{eq:eta2}
\end{align}
for some signs $\eta_{ij},\eta_{ij}'$. Subtracting \eqref{eq:010j} from \eqref{eq:eta1} would give a contradiction to $3$-independence if $\eta_{ij}=1$; hence $\eta_{ij}=-1$. Similarly, adding \eqref{eq:010j} to \eqref{eq:eta2} shows $\eta_{ij}'=1$. It thus follows (using \eqref{eq:010j} and \eqref{eq:eta1})
\[
\alpha_i + \alpha_j = \frac12(\gamma_{0i}-\gamma_{0i}' + \gamma_{0j} - \gamma_{0j}') = \gamma_{0j}-\gamma_{0i}' = \gamma_{ij}
\]
and
\[
\alpha_i-\alpha_j = \frac12(\gamma_{0i}-\gamma_{0i}' - \gamma_{0j}+\gamma_{0j}') = -\gamma_{0j}+\gamma_{0i} = \gamma_{ij}'
\]
as desired.
\end{proof}

And finally we consider $\Gamma_{\OO P^2}$.

\begin{lemma} Consider an orientable GKM$_3$-action with GKM graph $\Gamma_{\OO P^2}$. Then the induced labeling of the GKM graph is the same as that of a torus action on $\OO P^2$. 
\end{lemma}
\begin{proof}
  Let \(v_0,v_1,v_2\) be the vertices of \(\Gamma_M\).
  Denote by \(\alpha_i\) the weights at the edges between \(v_0\) and \(v_1\), by \(\beta_i\) the weights at the edges between \(v_1\) and \(v_2\), and by \(\gamma_i\) the weights at the edges between \(v_2\) and \(v_0\).

 Then by Sublemmas~\ref{sec:gkm-graphs-posit}, \ref{sec:gkm-graphs-posit-1} and \ref{sec:gkm-graphs-posit-2}, we may assume that the following relations hold:
 \begin{align}
   \label{eq:2}\gamma_1&=\alpha_1-\beta_1=\alpha_2+\beta_2\\
   \label{eq:3}\gamma_2&=\alpha_1-\beta_2=\alpha_2+\beta_1\\
   \label{eq:4}\gamma_3&=\alpha_1-\beta_3=-\alpha_2+\beta_4\\
   \label{eq:5}\gamma_4&=\alpha_1-\beta_4=-\alpha_2+\beta_3
 \end{align}
By adding equations \eqref{eq:3} and \eqref{eq:4} we get that
\begin{equation*}
  \alpha_1=\frac{1}{2}(\beta_1+\beta_2+\beta_3+\beta_4).
\end{equation*}
Hence, the \(\gamma_i\) are of the form \(\frac{1}{2}(-\beta_i+\sum_{j\neq i} \beta_j)\). By Sublemma~\ref{sec:gkm-graphs-posit-1} we may assume that
\begin{equation*}
  \gamma_1=\alpha_1-\beta_1=\alpha_2+\beta_2=\alpha_3+\beta_3=\alpha_4+\beta_4.
\end{equation*}
Thus, for \(i>1\),
\begin{equation*}
  \alpha_i=\frac{1}{2}(-\beta_1-\beta_i+\sum_{j\neq i, 1} \beta_j).
\end{equation*}
This proves the lemma.
\end{proof}

\section{Integer coefficients}\label{sec:integer-coefficients}

In this section we prove a version of our main theorem for integer coefficients. 
To do so we have to generalize some of the results from Section \ref{sec:gkm-theory} to integer coefficients.

The two ingredients we need for our theorem to hold are that the ordinary cohomology is encoded in the equivariant cohomology, i.e., that
\begin{equation}\label{eq:surjinteger}
H^*_T(M;\mathbb{Z})\longrightarrow H^*(M;\mathbb{Z})
\end{equation}
is surjective, and that the equivariant cohomology algebra is encoded in the combinatorics of the one-skeleton $M_1$, i.e., that a Chang-Skjelbred Lemma holds. This can be seen to be true only under additional assumptions on the isotropy groups of the action, see \cite[Corollary 2.2]{FranzPuppe}: if for all $p\notin M_1$ the isotropy group $T_p$ is contained in a proper subtorus of $T$, and $H^*_T(M;\mathbb{Z})$ is a free module over $H^*(BT)$, then
\[
0\longrightarrow H^*_T(M;\mathbb{Z}) \longrightarrow H^*_T(M^T;\mathbb{Z})\longrightarrow H^*_T(M_1,M^T;\mathbb{Z})
\]
is exact.

Because the fixed point set is always finite and \(M\) is orientable
in our situation, freeness of $H^*_T(M;\mathbb{Z})$ is equivalent to
$H^{\odd}(M;\mathbb{Z})=0$ by \cite[Lemma 2.1]{MasudaPanov}. Moreover,
under the assumption that $H^{\odd}(M;\mathbb{Z})=0$,
\eqref{eq:surjinteger} is surjective: this follows either from the
proof of \cite[Lemma 2.1]{MasudaPanov} or via the fact that in this
situation the Leray spectral sequence collapses.

Moreover, since \(H_T^*(M^T;\mathbb{Z})\rightarrow
H_T^*(M^T;\mathbb{R})\) is injective the formula (\ref{eq:7}) also
holds for the equivariant (integer) Pontryagin classes of \(M\). 

Denote by \(\mathbb{Z}_{\mathfrak{t}}^*\subset \mathfrak{t}^*\) the weight lattice of the torus \(T\).
Then we call two weights \(\alpha,\beta\in \mathbb{Z}_{\mathfrak{t}}^*\) coprime if there are primitive elements of \(\mathbb{Z}_{\mathfrak{t}}^*\), \(\alpha',\beta'\), and \(a,b\in\mathbb{Z}\) such that \(\alpha=a\alpha'\) and \(\beta=b\beta'\) and \(a\) and \(b\) are coprime.
Note that, if \(\alpha\neq 0\), \(a\) and \(\alpha'\) are uniquely determined up to sign by \(\alpha\).

\begin{lemma}
\label{lem:integer-coefficients}
 If, for an orientable GKM manifold with vanishing odd-degree integer cohomology, at each fixed point any two weights are coprime, then the Chang-Skjelbred Lemma holds for integer coefficients.
\end{lemma}
\begin{proof}
  Let \(p\) be a prime and denote by \(G\) the maximal \(\mathbb{Z}/p\mathbb{Z}\)-torus in \(T\).
  By \cite[Theorem 2.1]{FranzPuppe}, it is sufficient to show that \(M^G\) is contained in \(M_1\).
  
  At first note that by an iterated application of \cite[Theorem VII.2.2, p.~376]{Bredon}, we have that \(\dim H^{\odd}(M^G;\mathbb{Z}/p\mathbb{Z})\leq \dim H^{\odd}(M;\mathbb{Z}/p\mathbb{Z})=0\).
  Therefore every component of \(M^G\) is a $T$-invariant submanifold with non-trivial Euler characteristic. As the Euler characteristic always equals the Euler characteristic of the fixed point set, it follows that in all these components there is a $T$-fixed point $x$.
  
  Now consider the \(T\)-representation \(T_xM\).
  Then \(T_x(M^G)=(T_xM)^G\) is an invariant subrepresentation and therefore a direct sum of weight spaces \(V_\alpha\).
  Let \(q\in \mathbb{Z}\) and \(\alpha'\in \mathbb{Z}_{\mathfrak{t}}^*\) a primitive element such that \(\alpha=q\alpha'\).
  Then \(V_\alpha\) is fixed by \(G\) if and only if \(p\) divides \(q\).
  Since by assumption the weights at \(x\) are coprime, it follows that \(T_xM^G\) contains at most one weight space.
  Thus, the component of \(M^G\) which contains \(x\) is contained in a two-dimensional sphere fixed by a corank-one torus \(T'\) of \(T\).
  Hence \(M^G\) is contained in \(M_1\).
\end{proof}

\begin{remark}
  As pointed out to us by one of the referees,
  Lemma~\ref{lem:integer-coefficients} was conjectured for Hamiltonian
  torus actions of GKM type on symplectic manifolds by Tolman and
  Weitsman \cite{MR1736221}.
  This conjecture has been shown in \cite{schmid01:_cohom}  by Schmid under the assumption
  that all weights of the GKM graph are primitive vectors in \(\Z_{\mathfrak{t}}^*\).
\end{remark}

\begin{theorem}
\label{sec:integer-coefficients-1}
 Let $M$ be a positively curved orientable manifold with $H^{\odd}(M;\mathbb{Z})=0$ which admits an isometric torus action with finitely many fixed points such that
\begin{itemize}
\item At each fixed point any three weights of the isotropy representation are linearly independent and
\item At each fixed point any two weights are coprime.
\end{itemize}
 Then $M$ has the integer cohomology ring of a CROSS.
 Moreover, the total Pontryagin class of \(M\) is standard, i.e. there
 is a CROSS \(K\) and an isomorphism of rings
 \(f:H^*(M;\mathbb{Z})\rightarrow H^*(K;\mathbb{Z})\) such that \(f(p(M))=p(K)\).
\end{theorem}
\begin{proof}
  By Lemma~\ref{lem:integer-coefficients}, we have a GKM description of the equivariant cohomology of \(M\) with integer coefficients.
  Therefore the statement follows as in the proof of the second part of Theorem~\ref{sec:main-results-8}.
\end{proof}

\begin{remark} If $M$ is simply-connected and has the integer cohomology of $\mathbb{C} P^n$, then $M$ is already homotopy equivalent to $\mathbb{C} P^n$: choose a map $f:M\to K({\mathbb{Z}},2)=\mathbb{C} P^\infty$ such that the pullback of a generator of $H^2(\mathbb{C} P^\infty;\mathbb{Z})$ generates $H^2(M;\mathbb{Z})$. This map can be deformed to a map which takes values in the $2n$-skeleton of $\mathbb{C} P^\infty$, which is $\mathbb{C} P^n$, and this deformed map is then the desired homotopy equivalence. 
\end{remark}

\begin{remark} \label{rem:rathomtop}
  By \cite[Corollary 2.7.9]{MR1236839}, a simply connected manifold which has the
  same rational cohomology as a compact rank one symmetric space is
  formal in the sense of rational homotopy theory.
  Therefore it follows from   \cite[Theorem 2.2]{kreck-triantafillou} which is a generalization of \cite[Theorem
  12.5]{MR0646078}
 that up to diffeomorphism there are only finitely many simply connected integer cohomology \(\mathbb{K} P^n\)'s, \(\mathbb{K}=\C,\mathbb{H},\mathbb{O}\), with standard Pontryagin classes of dimension greater than four.
  Thus there are only finitely many diffeomorphism types of
  simply connected GKM$_3$-manifolds as in Theorem~\ref{sec:integer-coefficients-1}.
\end{remark}

\section{Non-orientable GKM-manifolds}
\label{sec:non-oriented}

In this section we prove an extension of Theorem~\ref{sec:main-results-8} to non-orientable GKM$_3$-manifolds.

\begin{lemma}
  If \(M^{2n}\) is a non-orientable GKM$_3$-manifold with an invariant metric of positive curvature, then the GKM graph of \(M\) coincides with the GKM graph of a linear torus action on \(\R P^{2n}\), i.e., it has only a single vertex.
\end{lemma}
\begin{proof}
  Denote by \(\tilde{M}\) the orientable double cover of \(M\).
  Then the torus action on \(M\) lifts to a torus action on \(\tilde{M}\).
  With this lifted action \(\tilde{M}\) is a GKM$_3$-manifold:   Indeed, it is obvious that the torus action on \(\tilde{M}\) has isolated fixed points.
  Moreover, the isotropy representation at a fixed point \(x\in \tilde{M}\) is isomorphic to the isotropy representation at \(p(x)\) where  \(p:\tilde{M}\rightarrow M\) is the covering map.
  Therefore the \(3\)-independence of the weights of \(T_x\tilde{M}\) follows.
  
  Now we show that \(H^{\odd}(\tilde{M};\R)=0\) which is equivalent to equivariant formality of the torus action on \(\tilde{M}\).
  Since the torus action on \(M\) is equivariantly formal and has only finitely many fixed points, we have \(H^{\text{odd}}(M;\mathbb{R})=0\).
  \(M\) is the quotient of a free \(\mathbb{Z}_2\)-action on \(\tilde{M}\).
  Therefore, by \cite[Chapter III]{Bredon}, \(H^*(M;\R)\) is isomorphic to \(H^*(\tilde{M};\R)^{\mathbb{Z}_2}\).
  Hence, it follows that \(\mathbb{Z}_2\) acts on \(H^{\text{odd}}(\tilde{M};\mathbb{R})\) and \(H^{2n}(\tilde{M};\mathbb{R})\) by multiplication with \(-1\).
  Now assume that  \(H^{\text{odd}}(\tilde{M};\R)\neq0\).
  Then, by Poincar\'e duality, there are \(\alpha_1,\alpha_2\in H^{\text{odd}}(\tilde{M};\mathbb{R})\) such that \(0\neq \alpha_1\alpha_2 \in H^{2n}(\tilde{M};\mathbb{R})\).
  But this is a contradiction to the description of the \(\mathbb{Z}_2\)-action given above.
  Therefore we must have \(H^{\text{odd}}(\tilde{M};\R)=0\).

  The one-skeleton of the action on \(\tilde{M}\) is a double covering of the one-skeleton of the action on \(M\).
  Therefore \(\Gamma_M\) is a quotient of a \(\mathbb{Z}_2\)-action on \(\Gamma_{\tilde{M}}\) which is free on the vertices. By the results in Section \ref{sec:main}, $\Gamma_{\tilde{M}}$ is one of the graphs described in Section \ref{sec:rankoneactions}.  It is easy to see that if \(\Gamma_{\tilde{M}}\) is not \(\Gamma_{S^{2n}}\), then every vertex of \(\Gamma_M\) is contained in an edge which contains two vertices and an edge which contains only one vertex.
  But this is impossible.
  Indeed, the only non-orientable \(T^2\)-manifold in dimension four, which admits an invariant metric of positive sectional curvature is \(\R P^4\) \cite{MR1255926}.
  Therefore there is no two-dimensional face of \(\Gamma_M\) which contains an edge which connects two vertices and an edge which contains only a single vertex.
  Hence, \(\Gamma_{\tilde{M}}\) must be isomorphic to \(\Gamma_{S^{2n}}\).
  In this case there is only a single vertex in \(\Gamma_M\).
  This implies that \(\Gamma_M=\Gamma_{\R P^{2n}}\).

  Therefore the lemma is proved.
\end{proof}

From the above lemma we get immediately the following corollary.

\begin{cor}
\label{sec:non-orientable-gkm}
  Let \(M\) be a compact positively curved non-orientable Riemannian manifold with $H^{\odd}(M;\R)=0$ which admits an isometric torus action of type GKM$_3$. 
 Then \(H^*(M;\R)=H^0(M;\R)=\R\).
\end{cor}

\appendix
\section{GKM actions on the nonsymmetric examples}\label{sec:nonsymmetricgkm}
Apart from the compact rank one symmetric spaces, the only known examples of even-dimensional positively curved manifolds are the homogeneous spaces $\SU(3)/T^2$, $\Sp(3)/\Sp(1)^3$ and $F_4/\Spin(8)$ \cite{Wallach}, and the biquotient $\SU(3)/\!/T^2$ \cite{Eschenburg}. Because these examples do not have the rational cohomology of a compact rank one symmetric space, Theorem \ref{sec:main-results-8} implies that they do not admit an isometric action of type GKM$_3$, but we will see in this section that all of them admit an isometric action of type GKM$_2$, and we will determine their GKM graphs.

The homogeneous examples admit a GKM$_2$-action by the general results of \cite{GuilleminHolmZara}: for any homogeneous space of the form $G/H$, where $G$ and $H$ are Lie groups of equal rank, the action of a maximal torus in $H$ (or $G$) is of type GKM$_2$. Let us first determine the GKM graphs of the three homogeneous examples above.
\subsection{$\SU(3)/T^2$} This example was considered in \cite{GuilleminHolmZara}, Section 5.3, where it was shown that the GKM graph of the $T^2$-isotropy action is the bipartite graph $K_{3,3}$. 
 \begin{figure}[htb]
 \includegraphics[width=117pt]{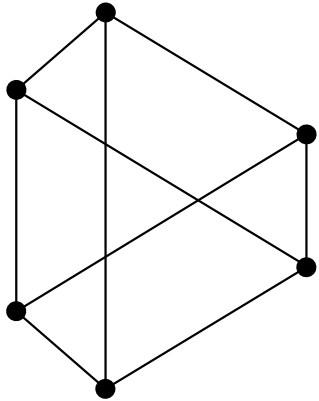}
 \caption{GKM graph of $\SU(3)/T^2$ and $\SU(3)/\!/T^2$}
 \label{figrp}
 \end{figure}
 Let us provide a slightly different argument which will turn out to be generalizable to the other two homogeneous examples. The homogeneous space $\SU(3)/T^2$ admits a homogeneous fibration
\[
S^2\longrightarrow \SU(3)/T^2\longrightarrow \C P^2,
\]
which implies that the GKM graph in question projects onto the GKM graph of $\C P^2$ (a triangle), with fibers the GKM graph of $S^2$ (a line). Now by \cite{GuilleminHolmZara}, Theorem 2.4, the vertices of the graph are in one-to-one correspondence with the elements of the Weyl group $W(\SU(3))$ of $\SU(3)$, which is the symmetric group $S_3$. Moreover, by the same theorem, two vertices are not joined by an edge if the corresponding elements $w,w'$ in the Weyl group satisfy the condition that $w^{-1}w'$ is not of order two. Thus, two vertices are not joined by an edge if they correspond to elements of $W(\SU(3))=S_3$ whose orders are congruent modulo $2$. This leaves $K_{3,3}$ as the only possibility for the graph.
\subsection{$\Sp(3)/\Sp(1)^3$} Consider the GKM$_2$-action on $\Sp(3)/\Sp(1)^3$ of a maximal torus $T^3\subset \Sp(1)^3$. This homogeneous space admits a homogeneous fibration
\[
S^4\longrightarrow \Sp(3)/\Sp(1)^3\longrightarrow \HH P^2,
\]
so the graph projects to a triangle with each edge doubled, with biangles as fibers. The vertices are in one-to-one correspondence to elements of the Weyl group quotient $W(\Sp(3))/W(\Sp(1)^3)$, and because the Weyl group of $\Sp(1)^3$, $\Z_2^3$, is normal in the Weyl group $\Z_2^3\rtimes S_3$ of $\Sp(3)$, with quotient $S_3$, similar considerations as above hold true. We conclude that the GKM graph is the bipartite graph $K_{3,3}$, with each edge doubled.
 \begin{figure}[htb]
 \includegraphics[width=117pt]{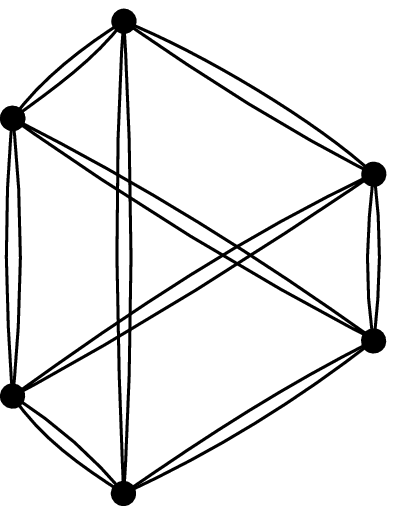}
 \caption{GKM graph of $\Sp(3)/\Sp(1)^3$}
 \label{figrp}
 \end{figure}
\subsection{$F_4/\Spin(8)$} Also this example admits a homogeneous fibration:
\[
S^8\longrightarrow F_4/\Spin(8)\longrightarrow \OO P^2
\]
which implies that the GKM graph projects onto a triangle with each edge replaced by four edges, with fibers the graph with two vertices and four edges. Again, the Weyl group of $\Spin(8)$ is normal in the Weyl group of $F_4$, with quotient $S_3$, see \cite{Adams}, Theorem 14.2. The GKM graph is therefore $K_{3,3}$, with each edge replaced by four edges.
 \begin{figure}[htb]
 \includegraphics[width=117pt]{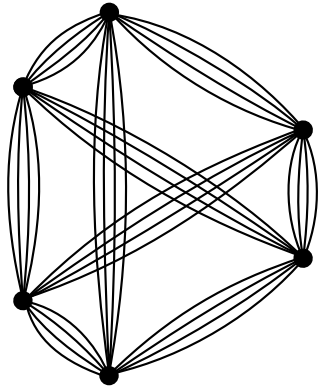}
 \caption{GKM graph of $F_4/\Spin(8)$}
 \label{figrp}
 \end{figure}
\subsection{$\SU(3)/\!/T^2$}
Finally, we consider the biquotient $\SU(3)/\!/T^2$. It is the quotient of $\SU(3)$ by the free $T^2$-action given by
\[
(t,w)\cdot g = \diag(t,w,tw)\cdot g\cdot \diag(1,1,t^{-2}w^{-2}).
\]
We denote the projection $\SU(3)\to \SU(3)/\!/T^2$ by $g\mapsto [g]$. This manifold admits an action of a two-dimensional torus by
\[
(a,b)\cdot [g] = [\diag(a,a,a^{-2})\cdot g\cdot \diag(b,b^{-1},1)]
\]
The six fixed points of the action are given by the elements $[g]$, where $g\in \SU(3)$ is a matrix that has (maybe after multiplying with \(-1\)) as column vectors a permutation of the standard basis vectors $e_1,e_2,e_3$. This action is of type GKM$_2$; two fixed points $[g_1],[g_2]$, with $g_i$ as above, are connected by a two-dimensional isotropy submanifold if and only if $g_1$ and $g_2$ have one identical column. For example, $\left[\left(\begin{matrix}
1 & 0 & 0 \\
0 & 1 & 0\\
0 & 0 & 1
\end{matrix}\right)\right]$ and $\left[-\left(\begin{matrix}
1 & 0 & 0 \\
0 & 0 & 1\\
0 & 1 & 0
\end{matrix}\right)\right]$ are joined by the projection of the $T^2$-invariant submanifold ${\mathrm{S}}(\U(1)\times \U(2))=\{\left(\begin{matrix} a & 0 & 0 \\ 0 & b & c \\0 & d & e \end{matrix}\right) \in \SU(3)\}$ of $\SU(3)$ to $\SU(3)/\!/T^2$. 

It follows that the GKM graph of $\SU(3)/\!/T^2$ is again the bipartite graph $K_{3,3}$: the vertices corresponding to the elements $\left(\begin{matrix}
1 & 0 & 0\\
0 & 1 & 0\\
0 & 0 & 1
\end{matrix}\right)$, $\left(\begin{matrix}
0 & 1 & 0\\
0 & 0 & 1\\
1 & 0 & 0
\end{matrix}\right)$ and $\left(\begin{matrix}
0 & 0 & 1\\
1 & 0 & 0\\
0 & 1 & 0
\end{matrix}\right)$ are pairwise not connected, and also not the vertices corresponding to $-\left(\begin{matrix}
1 & 0 & 0\\
0 & 0 & 1\\
0 & 1 & 0
\end{matrix}\right)$, $-\left(\begin{matrix}
0 & 1 & 0\\
1 & 0 & 0\\
0 & 0 & 1
\end{matrix}\right)$ and $-\left(\begin{matrix}
0 & 0 & 1\\
0 & 1 & 0\\
1 & 0 & 0
\end{matrix}\right)$.
\bibliography{pos_sec_gkm}{}
\bibliographystyle{amsplain}

\end{document}